\documentclass[11pt,a4paper]{amsart}

\usepackage{amsmath,amsthm,amssymb,amsfonts,mathrsfs,bbm}

\usepackage{graphicx} 
\usepackage{booktabs} 
\usepackage{lmodern} 
\usepackage{enumitem} 
\usepackage{multicol} 
\usepackage{setspace} 
\usepackage{tikz-cd} 
\usepackage{url} 
\usepackage{color,xcolor} 
\usepackage{comment}

\allowdisplaybreaks 

\newtheorem{thm}{Theorem}[section]
\newtheorem{prop}[thm]{Proposition}
\newtheorem{cor}[thm]{Corollary}
\newtheorem{lemma}[thm]{Lemma}
\newtheorem{qst}[thm]{Question}

\theoremstyle{definition}
\newtheorem{defn}[thm]{Definition}
\newtheorem*{defn*}{Definition}
\newtheorem{rmrk}[thm]{Remark}

\newtheorem{claim}{Claim}

\newcommand{\ord}{\textnormal{Ord}}

\newcommand{\powerset}{\mathcal{P}}
\DeclareMathOperator{\dom}{dom}

\DeclareMathOperator{\supp}{supp}
\DeclareMathOperator{\crit}{crit}
\DeclareMathOperator{\cf}{cf}

\newcommand{\p}{\mathbb{P}}
\newcommand{\q}{\mathbb{Q}}
\newcommand{\bc}{\mathbb{C}}

\newcommand{\boldt}{\mathbb{T}}
\newcommand{\dotq}{\dot{\mathbb{Q}}}
\newcommand{\dotp}{\dot{\mathbb{P}}}

\newcommand{\dotr}{\dot{\mathbb{R}}}

\newcommand{\lessd}{{<}\delta}

\newcommand{\kl}{(\kappa,\lambda)}

\DeclareMathOperator{\Add}{Add}

\title{Woodin for strong compactness cardinals}

\author{Stamatis Dimopoulos}
\address{School of Mathematics, University of Bristol, University Walk, Bristol, BS8 1TW, UK}
\email{stamatiosdimopoulos@gmail.com}
\thanks{The author is grateful to Arthur Apter and Philipp Schlincht for the very helpful conversations, that shaped the topic of this paper.}

\begin{document}

\begin{abstract}
Woodin and Vop\v enka cardinals are established notions in the large cardinal hierarchy and it is known that Vop\v enka cardinals are the Woodin analogue for supercompactness. Here we give the definition of Woodin for strong compactness cardinals, the Woodinised version of strong compactness, and we prove an analogue of Magidor's identity crisis theorem for the first strongly compact cardinal.
\end{abstract}

\maketitle

\section{Introduction} In \cite{magidor-identity-crises} Magidor established the ``identity crisis" of the first strongly compact cardinal, which can consistently be the first measurable or the first supercompact cardinal. This is by now a classic result in set theory and actually created a new field 
studying the ``identity crises" that accompany concepts related to strong compactness.\footnote{See \cite{identity-crises-I}, \cite{identity-crises-II} and \cite{identity-crises-III} for a very small sample of results in this area.} 
We further contribute to this area by establishing another identity crisis, to a concept created by combining Woodin and strongly compact cardinals.

Woodin and Vop\v enka cardinals, although originally defined in different context and for different reasons, are quite similar. 
A cardinal $\delta$ is Woodin if one of the following two equivalent definitions hold:
\begin{enumerate}
	\item for every function $f:\delta\to \delta$ there is $\kappa<\delta$ which is a closure point of $f$ and there is an elementary embedding $j:V\to M$ with critical point $\kappa$ and $V_{j(f)(\kappa)}\subseteq M$,
	\item for every $A\subseteq V_\delta$ there is a cardinal $\kappa<\delta$ which is $\lessd$-strong for $A$.
	
\end{enumerate} 
It was already known (see 24.19 in \cite{kanamori}) that replacing strongness by supercompactness in (2) we obtain a notion equivalent to Vop\v enka cardinals. Moreover, in \cite{perlmutter-2015}, Perlmutter showed that the same happens with (1) when we replace the clause $V_{j(f)(\kappa)}\subseteq M$ by ${}^{j(f)(\kappa)}M\subseteq M$. This makes Vop\v enka cardinals a Woodinised version for supercompact cardinals.

It is natural to consider what happens in (2) if we instead replace the strongness clause with a strong compactness clause, since strong compactness is an intermediate notion between strongness and supercompactness. In this article we look at this new type of cardinals, which we call \textit{Woodin for strong compactness}, and we explore their properties. For instance, we show that Woodin for strong compactness cardinals also have an equivalent definition which resembles (1), thus making them a reasonable Woodin analogue for strong compactness. The main result we establish is the identity crisis of the first Woodin for strong compactness cardinal. We show that it can consistently be the first Woodin or the first Woodin limit of supercompact cardinals.

The structure of the paper is as follows. In Section 2, we review some known facts about large cardinals and forcing. In Section 3, we give the definition of Woodin for strong compactness cardinals and show that they have properties similar to those of Woodin and Vop\v enka cardinals. Section 4 is split into two subsections, each dealing with one end of the identity crisis of the first Woodin for strong compactness cardinal. Finally, Section 5 includes some further results and open questions.

\section{Preliminaries} 
We will occasionally use interval notation $(\alpha,\beta)$ for two ordinals $\alpha<\beta$, to denote the set $\{\xi\mid \alpha<\xi<\beta\}$.

The large cardinal notions we deal with are witnessed by the existence of elementary embeddings of the form $j:V\to M$, where $V$ is the universe we work in and $M\subseteq V$ is a transitive class. The critical point of an elementary embedding $j$ is denoted by $\crit(j)$. For two cardinals $\kappa,\lambda$ we say $\kappa$ is \emph{$\lambda$-strong} if there is $j:V\to M$ with $\crit(j)=\kappa$, $j(\kappa)>\lambda$ and $V_\lambda\subseteq M$. We will also say $\kappa$ is \emph{${<}\mu$-strong} if it is $\lambda$-strong for all $\lambda<\mu$. We will always assume that $\lambda\geq\kappa$ even when not mentioned explicitly. 

Similarly, we have the concepts of a $\lambda$-supercompact and ${<}\mu$-supercompact cardinal $\kappa$. In this case we isolate the concept of a \emph{$\lambda$-supercompactness embedding} which is an elementary embedding $j:V\to M$ with $\crit(j)=\kappa$, such that $j$ is the ultrapower embedding by a normal ultrafilter on $\powerset_\kappa \lambda$.

A cardinal $\kappa$ is called \emph{$\lambda$-strong for $A$}, where $A$ is any set, if there is a $\lambda$-strongness embedding $j:V\to M$ with $\crit(j)=\kappa$, satisfying the property $A\cap V_\lambda=j(A)\cap V_\lambda$. Analogously, $\kappa$ is \emph{$\lambda$-supercompact for $A$} if there is a $\lambda$-supercompactness embedding $j:V\to M$ with $\crit(j)=\kappa$ and $A\cap V_\lambda=j(A)\cap V_\lambda$. Once again, we use expressions like $\kappa$ is ${<}\mu$-strong for $A$ to mean that $\kappa$ is $\lambda$-strong for $A$ for all $\lambda<\mu$, and it is always assumed that $\lambda\geq\kappa$.

We will make use of the following known result.

\begin{prop}\label{prop:strong-cohere}
Suppose $\kappa\leq \lambda <\mu$, $\kappa$ is ${<}\lambda$-strong (for $A$) and $\lambda$ is $\mu$-strong (for $A$). Then $\kappa$ is $\mu$-strong (for $A$).
\end{prop}

\begin{proof}
Let $j_1:V\to M$ be a $\mu$-strongness embedding with $\crit(j_1)=\lambda$ and $j_1(\lambda)>\mu$. By elementarity, $\kappa$ is ${<}j_1(\lambda)$-strong in $M$ and in particular, $\mu$-strong. Hence, there is a $\mu$-strongness embedding $j_2:M\to N$ with $\crit(j_2)=\kappa$ and $j_2(\kappa)>\mu$. The composition $j:=j_2\circ j_1$ has $\crit(j)=\kappa$, $j(\kappa)>\mu$ and is $\mu$-strong. If we assume that $j_1$ is $\mu$-strong for $A$ and that $j_2$ is $\mu$-strong for $j_1(A)$, then $j$ will also be $\mu$-strong for $A$.
\end{proof}

Since strongness is captured by extenders, the following fact will be useful.

\begin{prop}\label{prop:closure-extender}
Suppose $E$ is a $(\kappa,\lambda)$-extender such the corresponding embedding $j_E:V\to M_E$ satisfies $V_\lambda\subseteq M_E$. If $\cf(\lambda)>\kappa$, then ${}^\kappa M_E\subseteq M_E$.
\end{prop}

\begin{proof}
Suppose $\langle x_\alpha:\alpha<\kappa\rangle$ is a sequence of elements of $M_E$. Noting that 
$$M_E=\{j(f)(a)| a\in [\lambda]^{<\omega}, f\colon [\kappa]^{|a|}\to V, f\in V\},$$
we can assume that for each $\alpha$, $x_\alpha=j(f_\alpha)(a_\alpha)$, for some $f_\alpha,a_\alpha$. Each $a_\alpha$ is in $V_\lambda$ and as $\lambda$ has cofinality greater than $\kappa$, the whole sequence $\langle a_\alpha:\alpha<\kappa\rangle$ is in $V_\lambda$ and consequently in $M_E$. Also, $\langle j(f_\alpha):\alpha<\kappa\rangle\in M$ because $\langle j(f_\alpha):\alpha<\kappa\rangle=j(\langle f_\alpha:\alpha<\kappa\rangle)\restriction \kappa$. Hence, $\langle x_\alpha:\alpha<\kappa\rangle =\langle j(f_\alpha)(a_\alpha):\alpha<\kappa\rangle\in M$.
\end{proof}

An elementary embedding $j:V\to M$ with $\crit(j)=\kappa$ and $j(\kappa)>\lambda$ is said to satisfy the \emph{weak $\lambda$-covering property} if there is $s\in M$ such that $j``\lambda\subseteq s$ and $M\models |s|<j(\kappa)$. We also say that $j$ satisfies the \emph{$\lambda$-covering property} if for any set $X\subseteq M$ with $|X|\leq \lambda$ there is $s\in M$ such that $X\subseteq s$ and $M\models |s|<j(\kappa)$. A cardinal $\kappa$ is \emph{$\lambda$-strongly compact} if there is $j:V\to M$ with $\crit(j)=\kappa$ that satisfies the weak $\lambda$-covering property. If $j$ also satisfies the $\lambda$-covering property, then it will be called a \emph{$\lambda$-strong compactness embedding}. For a set of ordinals $A$, $\kappa$ is $\lambda$-strongly compact for $A$ if there is a $\lambda$-strong compactness embedding $j:V\to M$ with $\crit(j)=\kappa$, satisfying the property $A\cap \lambda=j(A)\cap \lambda$. As before, expressions like \emph{$\kappa$ is ${<}\mu$-strongly compact} or \emph{${<}\mu$-strongly compact for $A$} mean that the property holds for all $\kappa\leq\lambda<\mu$.

We will see that Woodin for strong compactness cardinals, naturally imply the existence of cardinals which are both strongly compact and strong. We will use the following fact, which was suggested by the anonymous referee and simplifies a lot of the original arguments of the author's exposition.

\begin{prop}\label{prop:sc-str}
If $\kappa$ is both $\lambda$-strong and $\lambda$-strongly compact for some $\lambda\geq\kappa$, then there is an elementary embedding $j:V\to M$ with $\crit(j)=\kappa$, $j(\kappa)>\lambda$, $V_\lambda\subseteq M$, satisfying the weak $\lambda$-covering property. Furthermore, if $\kappa$ is also $\lambda$-strong for $A$ for some set $A$, then the embedding $j$ can also satisfy $A\cap V_\lambda=j(A)\cap V_\lambda$. 
\end{prop}

\begin{proof}
Since $\kappa$ is $\lambda$-strong, let $j_1:V\to M$ be a $\lambda$-strongness embedding with $\crit(j_1)=\kappa$. By elementarity, $j_1(\kappa)$ is $j_1(\lambda)$-strongly compact in $M$, so there is an elementary embedding $j_2:M\to N$ with $\crit(j_2)=j_1(\kappa)$, which satisfies the weak $j_1(\lambda)$-covering property. Now, if we let $j:=j_2\circ j_1$, it is easy to see that $\crit(j)=\kappa$ and $j(\kappa)>\lambda$ and since the critical point of $j_2$ is above $\lambda$, $V_\lambda\subseteq N$. Also, $j``\lambda\subseteq j_2``j_1(\lambda)$ and since the latter is covered by a set $s\in N$ of size less than $j_2(j_1(\kappa))$, it follows that $j$ has the weak $\lambda$-covering property.

Finally, if we had assumed that $j_1$ also has the property $j_1(A)\cap V_\lambda=A\cap V_\lambda$, for some set $A$, then using the fact that $\crit(j_2)>\lambda$ we can easily see that $j(A)\cap V_\lambda=A\cap V_\lambda$.
\end{proof}

\begin{rmrk}\label{rmrk:full-covering}
If we make a better choice of embeddings in the previous proof, we can actually guarantee the $j$ will satisfy the full $\lambda$-covering property. Namely, suppose $j_1$ is given by a $(\kappa,\mu)$-extender for some cardinal $\mu$ and that $j_2$ is an ultrapower embedding by a fine $M$-ultrafilter on $(\powerset_{j(\kappa)}j(\lambda))^M$. This implies that $j_2$ satisfies the full $j_1(\lambda)$-covering property in $M$ and that $N\subseteq M$. If we consider now a set $X\subseteq N$ such that $|X|\leq \lambda$, it follows that $X\subseteq M$. Also, using the extender formulation of $M$, we can write $X$ as $\{j_1(f_\alpha)(a_\alpha)\mid \alpha<\lambda\}$ for some sets $a_\alpha\in [\mu]^{<\omega}$ and some functions $f_\alpha:[\kappa]^{|a_\alpha|} \to V$. Clearly, there is a set $Y\in M$ such that $|Y|\leq j(\lambda)$ and $\{j(f_\alpha)\mid \alpha<\lambda\}\subseteq Y$ (just by taking $Y=j(\{f_\alpha\mid \alpha<\lambda\}))$. Using the covering property of $j_2$, we can cover $Y$ in $N$ with a set $s$ such that $|s|^N<j(\kappa)$. Then, it is easy to induce a cover of $X$ from $s$ that is still of size less than $j(\kappa)$ in $N$. 
\end{rmrk}

\begin{cor}\label{cor:sc-str}
Suppose $\kappa$ is a cardinal, $\lambda\geq\kappa$ and $A$ is a set of ordinals. If $\kappa$ is both $\lambda$-strong for $A$ and $\lambda$-strongly compact, then $\kappa$ is $\lambda$-strongly compact for $A$.
\end{cor}

It will be useful to review the usual characterisations of Woodin and Vop\v enka cardinals. For the proofs, see 24.19 and 26.14 in \cite{kanamori}. 

\begin{prop}\label{prop:defns-woodin}
The following are equivalent for a cardinal $\delta$.
\begin{enumerate}
\item $\delta$ is Woodin, i.e. for every function $f:\delta\to \delta$ there is $\kappa<\delta$ which is a closure point of $f$ and there is an elementary embedding $j:V\to M$ with $\crit(j)=\kappa$ and $V_{j(f)(\kappa)}\subseteq M$.
\item For every $A\subseteq V_\delta$, there is $\kappa<\delta$ which is $\lessd$-strong for $A$.
\end{enumerate}
\end{prop}

\begin{prop}\label{prop:defns-vopenka}
The following are equivalent for a cardinal $\delta$.
\begin{enumerate}
\item $\delta$ is Vop\v enka, i.e. for every function $f:\delta\to \delta$ there is $\kappa<\delta$ which is a closure point of $f$ and there is an elementary embedding $j:V\to M$ with $\crit(j)=\kappa$ and $^{j(f)(\kappa)}M\subseteq M$.
\item For every $A\subseteq V_\delta$, there is $\kappa<\delta$ which is $\lessd$-supercompact for $A$.
\end{enumerate}
\end{prop}

Concerning the preservation of Woodin and Vop\v enka cardinals in forcing extensions, we will use the following results, which we state without proofs. The first can be found in \cite{brooke-taylor-indestructibility} and the second follows from folklore results (details can be found in \cite{woodin-cardinals-forcing}).

\begin{thm}[\cite{brooke-taylor-indestructibility}]\label{vopenka-ind}
Suppose $\delta$ is a Vop\v enka cardinal and $\p=\langle \p_\alpha,\dotq_\beta\mid \alpha\leq \delta,\beta<\delta\rangle$ is an Easton support $\delta$-iteration with the following properties:
\begin{enumerate}
	\item For all $\alpha<\delta$, $|\dotq_\alpha|<\delta$.
	\item For all $\alpha<\delta$, there is $\beta<\delta$ such that for all $\gamma\geq\beta$, $\Vdash_{\p_\gamma} \dotq_\gamma$ is $\alpha$-directed closed.
\end{enumerate}
Then $\delta$ remains Vop\v enka after forcing with $\p$.
\end{thm}

\begin{thm}\label{thm:woodin-indesructibility}
Suppose $\delta$ is a Woodin cardinal, GCH holds and $\p$ is an Easton support $\delta$-iteration which satisfies:
\begin{enumerate}
	\item $\p\subseteq V_\delta$,
	\item for each $A\subseteq V_\delta$ there is a $\lessd$-strong for $A$ cardinal $\kappa<\delta$ such that $\p_\kappa\subseteq V_\kappa$ and all stages of $\p$ greater or equal to $\kappa$ are forced to be at least $\kappa^+$-strategically closed.
\end{enumerate}
Then $\delta$ remains Woodin after forcing with $\p$.
\end{thm}

We will use the notion of \emph{width} of an embedding, found in \cite{cummings-handbook}.

\begin{defn}
An elementary embedding $j:V\to M$ is said to have \emph{width} $\leq\lambda$ for some ordinal $\lambda$, if every $x\in M$ can be written in the form $j(f)(a)$, for some set $a\in M$ and some function $f\in V$ with $|\dom(f)|\leq \lambda$.
\end{defn}

For instance, if $j$ is a (short) extender embedding with $\crit(j)=\kappa$, then it has width $\leq\kappa$. Also, if $j$ is the ultrapower embedding by an ultrafilter on $\powerset_\kappa\lambda$ for $\lambda\geq\kappa=\crit(j)$, then it has width $\leq\lambda^{<\kappa}$.

Concerning our notation on forcing, we follow closely \cite{cummings-handbook}. In particular, by $q\leq p$ we mean that $q$ is stronger than $p$ and by $\kappa$-distributive, we mean that the intersection of ${<}\kappa$-many dense open sets is open dense. For a forcing notion $\p$ we can define a game $G_\alpha(\p)$ of $\alpha$ many moves, where a player ODD playing at odd stages and a player EVEN playing at even stages, choose stronger and stronger conditions, with EVEN always starting with the trivial condition at $0$-stage. A forcing notion $\p$ is called \emph{$\kappa$-strategically closed} if player EVEN has a winning strategy in the game $G_\kappa(\p)$ and ${<}\kappa$-strategically closed if it is $\alpha$-strategically closed for all $\alpha<\kappa$.

We describe now two of the forcing notions that we will use. The first is the forcing which shoots a club of non-strong cardinals below an inaccessible cardinal $\kappa$. The poset we use is $\p=\{p\mid p$ is a closed bounded subset of $\kappa$, consisting of cardinals which are not ${<}\kappa$-strong$\}$, ordered by end-extension. It is easy to see that a generic filter for $\p$ induces a club subset of $\kappa$ consisting of cardinals which are not ${<}\kappa$-strong and that $\p$ is ${<}\kappa$-strategically closed and thus, $\kappa$-distributive. Moreover, it is $\kappa^+$-c.c and so, no cardinals are collapsed after forcing with $\p$. Note that if $\kappa$ is Woodin, then forcing with $\p$ destroys its Woodinness.

The second forcing we use is adding a non-reflecting stationary set at some given inaccessible cardinal $\kappa$, using cardinals of cofinality equal to some fixed regular $\lambda<\kappa$. We use the poset $\p$ whose conditions are functions $p:\alpha\to 2$, where $\alpha<\lambda$ and $p$ is the characteristic function of a (bounded) subset of $\kappa$, consisting of ordinals of cofinality $\lambda$, which is not stationary at its supremum and neither has any initial segment stationary at its supremum. The order is end-extension. Standard arguments show that $\p$ is $\kappa$-strategically closed and $\lambda$-directed closed. It is also $\kappa^+$-c.c., so it does not collapse any cardinals.

In our results we use Silver's criterion along with standard arguments to lift elementary embeddings through forcing. We mention here the two main techniques used in constructing the required generic filters, which can be found in \cite{cummings-handbook} or \cite{hamkins-largecardinals}.

\begin{prop}[Diagonalisation]\label{thm:diagonalisation}
Suppose $M\subseteq V$ is an inner model, $\p\in M$ is a forcing notion and $p\in \p$. If
\begin{enumerate}
\item ${}^\kappa M\subseteq M$
\item $\p$ is ${<}\kappa^+$-strategically closed in $M$
\item there are at most $\kappa^+$-many maximal antichains of $\p$ in $M$, counted in $V$,
\end{enumerate}
then there is in $V$, an $M$-generic filter $H\subseteq \p$ such that $p\in H$.
\end{prop}

\begin{prop}[Transferring]\label{thm:transering-generic}
Suppose $j:V\to M$ is an elementary embedding with width $\leq\lambda$ and let $\p$ be a $\lambda^+$-distributive forcing notion. If $G\subseteq \p$ is a $V$-generic filter, then the filter $H$ generated by $j``G$ is $M$-generic for $j(\p)$. 
\end{prop}

When forcing in the presence of large cardinals, it is many times useful to know that no new large cardinals are created. In \cite{hamkins-approximation}, Hamkins showed how such arguments work when a forcing iteration has low enough closure points. We write the definition of closure points and a summary of the results of \cite{hamkins-approximation} that we need in this article.

\begin{defn}
A forcing notion has a \emph{closure point} at $\alpha$ if it can be factorised as $\p\ast \dotq$, where $|\p|\leq \alpha$ and $\Vdash_\p \dotq$ is $(\alpha+1)$-strategically closed.
\end{defn}

\begin{thm}[\cite{hamkins-approximation}] \label{thm:cov-approx}
If $V\subseteq V[G]$ is a set forcing extension with closure point at $\alpha$ and $j : V [G] \to \bar{N}$ is a definable embedding in $V [G]$ with $V[G]\models {}^{\alpha}\bar{N}\subseteq \bar{N}$ and $\alpha<\crit(j)$, then the restriction $j\restriction V:V\to N$, where $N=\bar{N}\cap V$, is an elementary embedding, definable in $V$. Furthermore,
\begin{enumerate}
\item if $V_\lambda\subseteq \bar{N}$ for some $\lambda$, then $V_\lambda\subseteq N$;
\item if $V[G]\models {}^\lambda\bar{N}\subseteq \bar{N}$ for some $\lambda$, then $V\models {}^\lambda N\subseteq N$;
\item if $j$ is $\lambda$-strongly compact for some $\lambda$ and $V\subseteq V[G]$ satisfy the $\kappa$-covering property, i.e. for every set $s\in V[G]$ with $|s|^{V[G]}<\kappa$ and $s\subseteq V$ there is $s'\in V$ with $s\subseteq s'$ and $|s|^V<\kappa$, then $j\restriction V$ is also $\lambda$-strongly compact.
\end{enumerate}
\end{thm}

Finally, the following fact can be found in \cite{cummings-handbook}.

\begin{prop}\label{prop:lift-width}
Suppose $j^+:V[G]\to M[H]$ is the lift of an embedding $j:V\to M$, such that $j$ has width $\leq\lambda$. Then $j^+$ also has width $\leq\lambda$.
\end{prop}

\section{Woodin for strong compactness cardinals}

We define now the main concept of this article.

\begin{defn}
A cardinal $\delta$ is called \emph{Woodin for strong compactness} or \emph{Woodinised strongly compact} if for every $A\subseteq \delta$ there is $\kappa<\delta$ which is $\lessd$-strongly compact for $A$.
\end{defn}

The definition is obtained by replacing the strongness or supercompactness clause in (2) of \ref{prop:defns-woodin} or \ref{prop:defns-vopenka}, by a strong compactness clause. In this section, we will see that Woodinised strong compactness is a reasonable Woodin analogue. Firstly, we show that the definition implies inaccessibility.

\begin{prop}\label{prop:wfsc-inaccessible}
If $\delta$ is Woodin for strong compactness, then it is an inaccessible limit of $\lessd$-strongly compact cardinals.
\end{prop}

\begin{proof}
To show that $\delta$ must be regular, assume otherwise and let $\cf(\delta)=\kappa_0<\delta$. Fix an unbounded set $A\subseteq \delta$ such that $|A|=\kappa_0$ and $\min(A)>\kappa_0$, and let $\kappa$ be $\lessd$-strongly compact for $A$. Pick $\lambda\in(\kappa,\delta)$ such that $A\cap (\kappa,\lambda)$ is non-empty and let $j:V\to M$ be a $\lambda$-strong compactness for $A$ embedding with $\crit(j)=\kappa$. Since $A\cap \lambda=j(A)\cap \lambda$, it follows that $j(A)\cap j(\kappa)$ is non-empty and by elementarity, $A\cap \kappa$ is non-empty. However, since $\kappa$ is regular, $A\cap \kappa$ must be bounded by some $\alpha<\kappa$. By elementarity, $j(A)\cap j(\kappa)$ is also bounded by $j(\alpha)=\alpha<\kappa$. But then $j(A)\cap (\kappa,\lambda)=A\cap (\kappa,\lambda)$ should be empty, which is absurd.

If $\delta$ were a successor cardinal, say $\delta=\kappa^+$, then there would be no cardinal below $\delta$ which is $\lessd$-strongly compact for $A$, where $A=\kappa$. Thus, $\delta$ must be a limit cardinal.

If there was an ordinal $\alpha<\delta$ such that there are no $\lessd$-strongly compact cardinals in $[\alpha,\delta)$, then let $\kappa$ be $\lessd$-strongly compact for $B$, where $B=\alpha$. Pick $\lambda>\alpha$ and let $j:V\to M$ be a $\lambda$-strongly compact for $B$ embedding. Then, $B\cap \lambda=j(B)\cap \lambda$, but this is absurd since $B\cap \lambda=\alpha$ and $j(B)\cap \lambda=\lambda$, because $j(B)=j(\alpha)\geq j(\kappa)>\lambda$. Hence, $\delta$ must be a limit of $\lessd$-strongly compact cardinals which also implies that $\delta$ is a strong limit.
\end{proof}

The following Proposition is based on properties of Woodin cardinals (see Lemma 11 in \cite{cody-easton-woodin} for instance). We will use the following notation: \emph{$\kappa$ is $\lessd$-strongly compact for $A_1\oplus A_2$}, where $A_1,A_2\subseteq \delta$, if for all $\lambda\in(\kappa,\delta)$ there is a $\lambda$-strong compactness embedding $j:V\to M$ with $\crit(j)=\kappa$, such that $A_1\cap \lambda=j(A_1)\cap \lambda$ and $A_2\cap \lambda=j(A_2)\cap \lambda$.

\begin{prop}\label{prop:wfsc-second-order}
The following are equivalent for a cardinal $\delta$.
\begin{enumerate}
	\item For every $A\subseteq \delta$, there is $\kappa<\delta$ which is $\lessd$-strongly compact for $A$.
	\item For every $A_1,A_2 \subseteq \delta$, there is $\kappa<\delta$ which is $\lessd$-strongly compact for $A_1\oplus A_2$.
	\item For every $A\subseteq V_\delta$, there is a $\kappa<\delta$ which is $\lessd$-strongly compact and $\lessd$-strong for $A$.
\end{enumerate}
\end{prop}

\begin{proof}
$(1)$ and $(2)$ are clearly equivalent, since we can code two sets of ordinals using an absolute pairing function, such as the G\"odel pairing function. 

For $(2)\rightarrow (3)$, fix a set $A\subseteq V_\delta$ and using the fact that by \ref{prop:wfsc-inaccessible} $\delta$ is inaccessible, let $R$ be a relation on $\delta$ such that the Mostowski collapse $\pi:\langle \delta,R,\rangle \to \langle V_\delta,\in\rangle$ has the property that for every $\beth$-fixed point $\lambda<\delta$, $\pi\restriction \lambda:\langle \lambda,R\restriction \lambda\rangle\simeq \langle V_\lambda,\in\rangle$. Let $A_1=\{\langle \alpha,\beta\rangle_G \mid \langle \alpha,\beta\rangle\in R\}$ and $A_2=\pi^{-1}``A$. By our assumption, there is $\kappa$ which is $\lessd$-strongly compact and strong for $A_1\oplus A_2$, so for any $\beth$-fixed point $\lambda<\delta$ there is $j:V\to M$ with $\crit(j)=\kappa$, the weak $\lambda$-covering property, $A_1\cap \lambda=j(A_1)\cap \lambda$ and $A_2\cap \lambda=j(A_2)\cap \lambda$. The set $A_1\cap \lambda$ codes $R\restriction \lambda$, from which we can obtain $V_\lambda$. Thus, $V_\lambda\subseteq j(A_2)$. By elementarity, we also have that $\pi\restriction \lambda=j(\pi)\restriction \lambda$ and it is now easy to see that $A_2\cap V_\lambda=j(A_2)\cap V_\lambda$ implies $j(A)\cap V_\lambda=A\cap V_\lambda$.

Finally, $(3)\rightarrow (1)$ follows easily from \ref{cor:sc-str}, so the proof is complete.
\end{proof}

It now follows that every Woodin for strong compactness cardinal is Woodin and every Vop\v enka cardinal is Woodin for strong compactness. However, the following result shows that any Woodin limit of supercompact cardinals is Woodin for strong compactness and there are plenty of such cardinals below any Vop\v enka cardinal.

\begin{prop}
Suppose $\delta$ is Woodin and there are unboundedly many $\lessd$-supercompact cardinals below $\delta$. Then $\delta$ is Woodin for strong compactness.
\end{prop}

\begin{proof}
Let $S\subseteq \delta$ denote the collection of $\lessd$-supercompact cardinals below $\delta$. Fix any $A\subseteq \delta$ and let $\kappa<\delta$ be a $\lessd$-strong for both $A$ and $S$ (not necessarily witnessed by a single embedding). Then, $S\cap \kappa$ is unbounded and the usual proof of Menas' result, shows that $\kappa$ must be $\lessd$-strongly compact. By \ref{cor:sc-str}, it follows that $\kappa$ is $\lessd$-strongly compact for $A$ and as $A$ was chosen arbitrarily, $\delta$ is Woodin for strong compactness.
\end{proof}

In the following result, we provide further characterisations of Woodinised strong compactness, analogous to (1) of \ref{prop:defns-woodin} and \ref{prop:defns-vopenka}.

\begin{thm} \label{thm:wfsc-definable}
The following are equivalent for a cardinal $\delta$.
\begin{enumerate}
	\item $\delta$ is Woodin for strong compactness.
	\item For every function $f:\delta\to \delta$ there is $\kappa<\delta$ which is a closure point of $f$ and there is an elementary embedding $j:V\to M$ with $\crit(j)=\kappa$, $V_{j(f)(\kappa)}\subseteq M$ and $j$ satisfies the $j(f)(\kappa)$-covering property.
	\item For every function $f:\delta\to \delta$ there is $\kappa<\delta$ which is a closure point of $f$ and there is an elementary embedding $j:V\to M$ with $\crit(j)=\kappa$, $V_{j(f)(\kappa)}\subseteq M$, $j$ satisfies the $j(f)(\kappa)$-covering property, and $j$ is generated by an extender $E\in V_\delta$ and a fine ultrafilter on $\powerset_\kappa\lambda$ for some $\lambda<\delta$.
\end{enumerate}
\end{thm}

\begin{proof}

The proof is based on the corresponding arguments for Woodin cardinals, such as Lemma 34.2 \cite{jech} or Theorem 24.16 in \cite{kanamori}. 

To show $(1)\rightarrow (2)$, fix a function $f:\delta\to \delta$ and apply $(3)$ of \ref{prop:wfsc-second-order} for $A=f$ to fix a $\kappa$ which is $\lessd$-strongly compact and $\lessd$-strong for $f$. Pick $\lambda>f(\kappa)$ and let $j:V\to M$ be an elementary embedding with $\crit(j)=\kappa$, $j(\kappa)>\lambda$, $V_\lambda\subseteq M$, satisfying the $\lambda$-covering property and $f\cap V_\lambda=j(f)\cap V_\lambda$. Note that the last condition implies that $j(f)(\kappa)=f(\kappa)<\lambda$ and so $V_{j(f)(\kappa)}\subseteq M$. Also, for each $\alpha<\kappa$, $j(f)(\alpha)=f(\alpha)<\lambda<j(\kappa)$ and so, $f(\alpha)<\kappa$. Thus, $f``\kappa\subseteq \kappa$ and the proof is complete. 

For $(2)\rightarrow (3)$, fix $f:\delta\to\delta$ and let $g:\delta\to\delta$ be a function such that $g(\alpha)$ is an inaccessible cardinal above $f(\alpha)$. By our assumption there is $\kappa<\delta$ which is a closure point of $g$ and an embedding $j:V\to M$ with $\crit(j)=\kappa$, $V_{j(g)(\kappa)}\subseteq M$ and $j$ satisfies the $j(g)(\kappa)$-covering property. From $j$ we can derive a $(\kappa,j(g)(\kappa))$-extender $E$ and a fine ultrafilter $U$ on $\powerset_\kappa j(g)(\kappa)$. For simplicity let $\lambda=j(g)(\kappa)$. By elementarity, $\lambda$ is inaccessible in $M$ and since $V_{\lambda}\subseteq M$ it is inaccessible in $V$ too. Thus, the extender embedding $j_E:V\to M_E$ is $\lambda$-strong and has critical point $\kappa$. Also, $j_E(U)$ is a fine ultrafilter on $(\powerset_{j_E(\kappa)}j_E(\lambda))^{M_E}$, so the ultrapower embedding $k:M_E\to M_{j_E(U)}$ has $\crit(k)=j_E(\kappa)$ and satisfies the $j_E(\lambda)$-covering property. Now, as in \ref{cor:sc-str}, $j^*:=j_U\circ k$ is both $\lambda$-strong and $\lambda$-strongly compact and it is easy to see that $j^*(g)(\kappa)\leq \lambda$. (3) now follows since $\kappa$ is a closure point of $f$ and by elementarity, $j^*(f)(\kappa)<j^*(g)(\kappa)$.

$(3)$ trivally implies $(2)$ so it remains to show that $(2)$ implies $(1)$. Fix $A\subseteq \delta$ and let $f:\delta \to \delta$ be the function defined as follows. If $\alpha$ is $\lessd$-strongly compact for $A$, then let $f(\alpha)=0$. Otherwise, let $f(\alpha)$ be an inaccessible cardinal $\gamma$ greater than $\beta$, where $\beta<\delta$ is least such that $\alpha$ is not $\beta$-strongly compact for $A$. By our assumption, there is $\kappa<\delta$ such that $f``\kappa\subseteq \kappa$ and there is $j:V\to M$ with $\crit(j)=\kappa$, satisfying the $\lambda$-covering property and $V_{j(f)(\kappa)}\subseteq M$. Now it suffices to show that $\kappa$ is $<j(\delta)$-strongly compact for $j(A)$ in $M$, as elementarity will give the desired conclusion.

If this is not the case, then by the definition of $f$ there is some $\beth$-fixed point $\lambda<j(f)(\kappa)$ such that $\kappa$ is not $\lambda$-strongly compact for $j(A)$ in $M$. Note that $j(\kappa)$ is a closure point of $j(f)$ and so, $\lambda<j(f)(\kappa)<j(\kappa)$. Since $j$ satisfies the $j(f)(\kappa)$-covering property, it also satisfies the $\lambda$-covering property. From $j$, we can derive a $(\kappa,\lambda)$-extender $E$ and a fine ultrafilter $U$ on $\powerset_\kappa\lambda$. Since $j(f)(\kappa)$ is inaccessible in $M$, it follows that $E,U\in M$.

The arguments for the case of Woodin cardinals, show that using $E$ in $M$, we get an extender embedding $j_E:M\to N$ which is $\lambda$-strong for $j(A)$. In our case, we also have a fine ultrafilter on $\powerset_\kappa\lambda$ so $\kappa$ is both $\lambda$-strong for $j(A)$ and $\lambda$-strongly compact in $M$. It follows by \ref{cor:sc-str} that $\kappa$ is $\lambda$-strongly compact for $j(A)$ in $M$, which is a contradiction.
\end{proof}	

\medskip

These characterisations show that Woodinised strong compactness is a reasonable Woodin-like concept. (3) is not used in later arguments but it is worth noting that it is a $\Pi_1^1$-definition, which shows that the first Woodin for strong compactness cardinal is not even weakly compact.

As with the other Woodin-like cardinals, Woodin for strong compactness cardinals come equipped with a normal filter. Call a set $X\subseteq \delta$ \emph{Woodin for strong compactness in $\delta$} if for any $f:\delta\to \delta$ there is $\kappa\in X$ which is a closure point of $f$ and there is $j:V\to M$ with $\crit(j)=\kappa$ which satisfies the weak $j(f)(\kappa)$-covering property and $V_{j(f)(\kappa)}\subseteq M$. Let 
$$F=\{X\subseteq \delta\mid \delta-X\text{ is not Woodin for strong compactness in }\delta\}.$$

We can prove the following like in the case of Woodin or Vop\v enka cardinals.

\begin{prop}
$F$ is a (proper) filter on $\delta$ iff $\delta$ is Woodin for strong compactness.
\end{prop}

Note that a set $X\subseteq \delta$ is in $F$ iff there is a function $f:\delta\to\delta$ such that for each closure point $\kappa$ of $f$ for which there is an elementary embedding $j:V\to M$ with $\crit(j)=\kappa$, satisfying the $j(f)(\kappa)$-covering property and $V_{j(f)(\kappa)}\subseteq M$, $\kappa\in X$. This can be seen as the definition of a set $X\subseteq \delta$ being ``measure one" with respect to $F$, while the notion of being ``Woodin for strong compactness in $\delta$" can be seen as being ``positive" with respect to $F$. 

The proof of the following result follows the same arguments as in the Woodin case; see 26.15 in \cite{kanamori}.

\begin{prop}
Suppose $\delta$ is Woodin for strong compactness and $F$ is the associated filter. Then:
\begin{enumerate}
\item $F$ is normal.
\item For any $A\subseteq \delta$, $\{\alpha<\delta\mid \alpha$ is $\lessd$-strongly compact for $A\}\in F$.
\item For any $A\subseteq V_\delta$, $\{\alpha<\delta\mid \alpha$ is $\lessd$-strongly compact and strong for $A\}\in F$.
\item For any $X\in F$, $\{\alpha<\delta\mid \alpha$ is measurable and there is a normal ultrafilter $U$ on $\alpha$ such that $X\cap \alpha\in U\}\in F$.
\end{enumerate}
\end{prop}

\section{The first Woodin for strong compactness cardinal}

We now state the main result of the article, which is split in two theorems.

\begin{thm}\label{thm:first-woodin}
Suppose $\delta$ is a Vop\v enka cardinal. Then there is a forcing extension in which $\delta$ is Woodin for strong compactness and there are no Woodin cardinals below $\delta$.
\end{thm}

\begin{thm}\label{thm:first-vopenka}
Suppose $\delta$ is a Vop\v enka cardinal. Then there is a forcing extension inside which $\delta$ remains a Woodin limit of $\lessd$ supercompact cardinals (and so, Woodin for strong compactness) and there are no Woodin for strong compactness cardinals below $\delta$.
\end{thm}

These two results together establish the identity crisis of the first Woodin for strong compactness cardinal.

\begin{cor}
The first Woodin for strong compactness cardinal $\delta$ can consistently (modulo the existence of a Vop\v enka cardinal) be the first Woodin or the first Woodin limit of $\lessd$-supercompact cardinals.
\end{cor}

This can be seen as a Woodinised analogue of Magidor's original identity crisis theorem, which states that the first strongly compact can consistently be the first measurable or the first supercompact cardinal.

\subsection{Proof of Theorem \ref{thm:first-woodin}}

Suppose $\delta$ is a Vop\v enka cardinal. We define an Easton support $\delta$-iteration $\p=\langle \p_\alpha,\dotq_\beta\mid \alpha\leq\delta,\beta<\delta\rangle$. as follows. Let $\dotq_0$ be a name for $\Add(\omega,1)$ and if $\p_\alpha$ has been defined and $\alpha$ was Woodin in $V$, then let $\dotq_\alpha$ name the forcing which shoots a club of non ${<}\alpha$-strong cardinals below $\alpha$ (see Section 2). Otherwise, let $\dotq_\alpha$ name the trivial forcing. Let $G\subseteq \p$ be a $V$-generic filter.

Firstly, notice that since we forced with $\Add(\omega,1)$ in the first stage, we introduced a very low closure point. By \ref{thm:cov-approx} the forcing creates no new instances of strongness, thus there is no Woodin cardinal below $\delta$ in $V[G]$. Now, it remains to show why $\delta$ remains Woodin for strong compactness. This follow from a series of claims.

\begin{claim}\label{claim:strong-not-woodin}
In $V$, for every $A\subseteq V_\delta$ there is a cardinal $\kappa<\delta$ which is $\lessd$-strongly compact and $\lessd$-strong for $A$, but is not Woodin.
\end{claim}

\begin{proof}
In $V$, fix a set $A\subseteq V_\delta$ and let $S$ denote the collection of $\lessd$-supercompact cardinals below $\delta$. Let $\kappa$ be the first $\lessd$-strong for both $A$ and $S$ cardinal below $\delta$ in $V$. By this, we mean that for each $\lambda<\delta$ there are embeddings that witness the $\lambda$-strongness for $A$ and $\lambda$-strongness for $S$ of $\kappa$, without necessarily having one witnessing both properties. Since $S$ is unbounded in $\delta$, it is also unbounded in $\kappa$ and thus, $\kappa$ is a measurable limit of $\lessd$-strongly compact cardinals. The usual proof of Menas' result shows that $\kappa$ must be $\lessd$-strongly compact in $V$. 

We claim that $\kappa$ is not Woodin in $V$. Otherwise, by applying (2) of \ref{prop:defns-woodin} for $A\cap V_\kappa$ and $S\cap V_\kappa$, we could find a cardinal $\kappa_0<\kappa$ which is ${<}\kappa$-strong for both $A\cap V_\kappa$ and $S\cap V_\kappa$. By \ref{prop:strong-cohere}, $\kappa_0$ is $\lessd$-strong for both $A$ and $S$, which contradicts the choice of $\kappa$.
\end{proof}

\begin{claim}\label{claim:strong-v}
In $V[G]$, for every $A\subseteq (V_\delta)^V$, there is a cardinal $\kappa<\delta$ which $\lessd$-strong for $A$.
\end{claim}

\begin{proof}
Fix $A\subseteq V_\delta$ in $V$ and using Claim \ref{claim:strong-not-woodin}, let $\kappa<\delta$ be a cardinal which is $\lessd$-strong for $A$ and not Woodin. 

Pick $\lambda>\kappa$ such that $\lambda$ is inaccessible, $\p_\lambda\subseteq V_\lambda$ and $\lambda$ is not Woodin. Let $j:V\to M$ be a $\lambda$-strongness for $A$ embedding with $\crit(j)=\kappa$, $j(\kappa)>\lambda$ and $A\cap V_\lambda=j(A)\cap V_\lambda$. By our choice of $\lambda$, it is the case that $\p\cap V_\lambda=j(\p)\cap V_\lambda$. Moreover, we can assume that $j$ is an extender embedding so that by \ref{prop:closure-extender}, ${}^\kappa M\subseteq M$. Since $\p_\lambda\subseteq V_\lambda$ and $\lambda$ is inaccessible in both $V$ and $M$, it follows that the first $\lambda$-stages of $j(\p)$ are the same as those of $\p$. 

To lift $j$ through $\p$, we factorise it as $\p_\kappa\ast \dotp_{>\kappa}$, where $\dotp_{>\kappa}$ is a name for the stages greater than $\kappa$, noting that there is no forcing at $\kappa$. We start by lifting $j$ through $\p_\kappa$. Using the previous fact, $j(\p_\kappa)$ can be factorised as as $\p_\lambda\ast \dotp_{tail}$. We can use $G_\lambda$ as an $M$-generic filter for $\p_\lambda$ and we need to construct an $M[G_\lambda]$-generic filter for $\p_{tail}:=(\dotp_{tail})_{G_\lambda}$. By the definition of $\p$, it follows that $\p_{tail}$ is (much more than) $\kappa^+$-strategically closed in $M[G_\lambda]$. Since $\kappa$ is Mahlo and we are using Easton support, $\p_\kappa$ has the $\kappa$-c.c. and so, $V[G_\kappa]\models {}^\kappa M[G_\kappa]\subseteq M[G_\kappa]$. Since there is no forcing at stage $\kappa$, the stages of $\p_\lambda$ above $\kappa$ are $\kappa^+$-distributive in both $V[G_{\kappa}]$ and $M[G_{\kappa}]$. By standard arguments we have $V[G_\lambda]\models {}^\kappa M[G_\lambda]\subseteq M[G_\lambda]$. 

In order to construct an $M[G_\lambda]$-generic filter for $\p_{tail}=(\dotp_{tail})_{G_\lambda}$, we consider the structure 
$$X=\{j(f)(\kappa,\lambda)\mid f:[\kappa]^2\to V,f\in V\}.$$
With standard arguments it can be shown that $X$ is an elementary substructure of $M$ that contains the range of $j$ and that $V\models {}^\kappa X \subseteq X$. Also, $\kappa,\lambda\in X$ and so, $\p_\lambda,\dotp_{tail}\in X$. If we form $X[G_\lambda]$\footnote{By $X[G_\lambda]$ we denote the interpretation of all $\p_\lambda$-names in $X$ under $G_\lambda$.}, then it can be shown that $X[G_\lambda]$ is an elementary substructure of $M[G_\lambda]$ and that $V[G_\lambda]\models {}^\kappa X[G_\lambda]\subseteq X[G_\lambda]$. The point of using $X[G_\lambda]$ is that every name in $X$ for an antichain in $\dotp_{tail}$ has the form $j(f)(\kappa,\lambda)$ for some function $f:[\kappa]^2\to V_{\kappa+1}$. Hence, by our assumption of GCH there are at most $\kappa^+$-many maximal antichains of $\p_{tail}$ in $X[G_\lambda]$, as counted in $V[G_\lambda]$. Thus, the conditions of \ref{thm:diagonalisation} hold and we can construct in $V[G_\lambda]$ an $X[G_\lambda]$-generic filter $H_1\subseteq \p_{tail}$. 

Note that $H_1$ is also $M[G_\lambda]$-generic for $\p_{tail}$. To see this, let $D\subseteq \p_{tail}$ be an open dense set in $M[G_\lambda]$. Then, there is a $\p_\lambda$-name $\dot{D}\in M$ for $D$ and using the extender representation of $j$, we can write $\dot{D}$ as $j(f_D)(a)$, for some $a\in [\lambda]^{<\omega}$, $f_D:[\kappa]^{|a|}\to V$. Consider the set
\begin{align*}
D'=\bigcap\{(j(f_D)(b))_{G_\lambda}\mid b\in [\lambda]^{<\omega}, j(f_D)(b)\text{ is a }\p_\lambda\text{-name for} & \text{ an open dense} \\
 & \text{ subset of }\p_{tail}\}.
\end{align*}
$D'$ is well-defined since $j(f_D)(a)$ is in the set above. Also, $D'$ is definable from $\lambda$, $G_\lambda$, $\p_{tail}$ and $j(f_D)$ and hence definable in $X$. Since $\p_{tail}$ is $j(\kappa)$-distributive and $D'$ is the intersection of at most $\lambda$-many open dense sets, $D'$ is an open dense set contained in $D$. Since $H_1$ intersects $D'$ it also intersects $D$, therefore it is $M[G_\lambda]$-generic. Since $j``G_\kappa=G_\kappa\subseteq G_\lambda\ast H_1$, we can lift $j$ to $j:V[G_\kappa]\to M[j(G_\kappa)]$, where $j(G_\kappa)=G_\lambda \ast H_1$.

To further lift $j$ though $\p_{>\kappa}:=(\dotp_{>\kappa})_{G_\kappa}$, note that since there is no forcing at $\kappa$, $\p_{>\kappa}$ is $\kappa^+$-strategically closed in $V[G_\kappa]$. Let $G_{>\kappa}$ be the part of $G$ corresponding to $\p_{>\kappa}$. As $j$ is an extender embedding, it has width $\leq\kappa$ and by \ref{thm:transering-generic}, the filter generated by $j``G_{>\kappa}$ is $M[j(G_\kappa)]$-generic for $j(\p_{>\kappa})$. Hence, we can lift $j:V[G]\to M[j(G)]$, where $j(G)=G_\lambda\ast H_1\ast H_2$. 

Note that $(V_\lambda)^{V[G_\lambda]}=V_\lambda[G_\lambda]\subseteq M[G_\lambda]\subseteq M[j(G)]$, thus $j$ is a $\lambda$-strongness embedding. Moreover, $A\cap V_\lambda=j(A)\cap V_\lambda$ because $A\in V$ and $j\restriction V$ had the same property. Since $\lambda$ can be chosen arbitarily large below $\delta$, we showed that $\kappa$ is $\lessd$-strong for $A$ in $V[G]$.
\end{proof}

\begin{claim}\label{claim:strong-sc-v}
In $V[G]$, for every $A\subseteq (V_\delta)^V$, there is a cardinal $\kappa<\delta$ which both $\lessd$-strongly compact and $\lessd$-strong for $A$.
\end{claim}

\begin{proof}
If we fix $A\subseteq V_\delta$, we can use the proofs of Claims \ref{claim:strong-not-woodin} and \ref{claim:strong-v} to find a cardinal $\kappa_0<\delta$, which is $\lessd$-strong for $A$ and a limit of $\lessd$-supercompact cardinals, such that $\kappa_0$ remains $\lessd$-strong for $A$ in $V[G]$. So, all we need to show is that $\kappa_0$ remains $\lessd$-strongly compact in $V[G]$. We do this by showing that every $\lessd$-supercompact cardinal below $\kappa_0$ remains $\lessd$-strongly compact.

Let $\kappa<\kappa_0$ be a $\lessd$-supercompact and fix $\lambda\in (2^\kappa,\delta)$ such that $\lambda$ is not Woodin and $\p_\lambda\subseteq V_\lambda$. Let $j_1:V\to M$ be a $\lambda$-supercompactness embedding with $\crit(j_1)=\kappa$. By standard arguments, $\kappa$ is ${<}j_1(\kappa)$-strong in $M$ and, so there is an elementary embedding $j_2:M\to N$ with $\crit(j_2)=\kappa$, $j_2(\kappa)>\lambda$ and $V_{\lambda}\subseteq N$. We can choose $j_2$ so that is given by  a $(\kappa,\lambda)$-extender and such that $\kappa$ is not $\lambda$-strong in $N$.

Now. if we let $j:=j_2\circ j_1\colon V\to N$ then $j$ is a $\lambda$-strong compactness embedding. To see this, let $X\subseteq N$ be a set of size at most $\lambda$. Since $M$ is closed under $\lambda$-sequences, $j_1``X\in M$ and $|j_1``X|^M<j_1(\kappa)$. By elementarity, $|j_2(j_1``X)|^N<j_2(j_1(\kappa))=j(\kappa)$ and clearly $j``X\subseteq j_2(j_1``X)$, so it is the required cover.

We aim to lift $j$ through $\p$ and for this end, we factorise $\p$ as $\p_\kappa\ast \dotq_\kappa \ast\dotp_{\kl}*\dotp_{>\lambda}$, where $\dotp_{\kl}$ is a $\p_\kappa\ast \dotq_\kappa$-name for the stages in the interval $(\kappa,\lambda)$ and $\dotp_{>\lambda}$ is a name for the later stages. Note that the $\lambda$-stage is trivial.

By the properties of $j_1$, the first $\lambda$-stages of $j_1(\p)$ are the same as those of $\p$. So, we can factorise $j_1(\p_\lambda)$ as $\p_\lambda\ast\dotp_{tail}\cong \p_\kappa\ast \dotq_\kappa\ast \dotp_{\kl}\ast\dotp_{tail}$, where $\dotp_{tail}$ is a name for the stages in $(\lambda,j_1(\delta))$. By elementarity, 
$$j(\p_\lambda)=j_2(j_1(\p))\cong j_2(\p_\kappa)\ast j_2(\dotq_\kappa)\ast  j_2(\dotp_{\kl})\ast j_2(\dotp_{tail}).$$

\subsection*{Constructing a generic for $j_2(\p_\kappa)$.} 
By the properties of $j_2$ it follows that the first $\lambda$-stages of $j_2(\p_\kappa)$ are the same as those of $\p$ and so, we can factorise $j_2(\p_\kappa)$ as $\p_\lambda*\dotq_{tail}$, where $\dotq_{tail}$ is a name for the stages in $(\lambda,j_2(\kappa))$. As $G_\lambda$ is $V$-generic, it is also $N$-generic so we can form $N[G_\lambda]$. In $N[G_\lambda]$, $\q_{tail}=(\dotq_{tail})_{G_\lambda}$ is (much more than) $\kappa^+$-strategically closed. Since $\kappa$ is Mahlo and we are using Easton support, $\p_\kappa$ has the $\kappa$-c.c. Also $\q_\kappa$ has the $\kappa^+$-c.c. in both $V[G_\kappa]$ and $M[G_\kappa]$ and $\p_{\kl}$ is $\kappa^+$-distributive in both $V[G_{\kappa+1}]$ and $M[G_{\kappa+1}]$. Hence, by the standard arguments we have $M[G_\lambda]\models {}^\kappa N[G_\lambda]\subseteq N[G_\lambda]$. Now, let 
$$X=\{j(f)(\kappa,\lambda)\mid f:[\kappa]^2,f\in V\}.$$
Using exactly the same arguments as in Claim \ref{claim:strong-v}, we can construct in $M[G_\lambda]$ an $X[G_\lambda]$-generic filter $H_1$ for $\dotq_{tail}$, which is also $N[G_\lambda]$-generic. Since $j``G_\kappa=G_\kappa\subseteq G_\lambda\ast H_1$, we can lift $j_2$ to $j_2:M[G_\kappa]\to N[j_2(G_\kappa)]$, where $j_2(G_\kappa)=G_\lambda \ast H_1$.

\medskip

\subsection*{Constructing a generic for $j_2(\q_\kappa)$.}
We need an $N[j_2(G_\kappa)]$-generic $H_2$ for $j_2(\q_\kappa)$ such that if $g$ is the part of $G$ that corresponds to $\q_\kappa$, $j_2``g\subseteq H_2$. 
If $C_\kappa=\bigcup g$ is the generic club added to $\kappa$ by $\q_\kappa$, then $C_\kappa$ consists of cardinals $\alpha<\kappa$ which are not ${<}\kappa$-strong in $V$. By elementarity,, $j(\alpha)=\alpha$ is not ${<}j_2(\kappa)$-strong in $N$. Also, $j_2$ was chosen so that $\kappa$ is not $\lambda$-strong in $N$, hence $q=C_\kappa\cup \{\kappa\}$ is a condition in $j_2(\q_\kappa)$.

Now, the structure $X$ comes in use again. In the previous argument we formed $X[G_\lambda][H_1]$. Since $M[G_\lambda]\models {}^\kappa X[G_\lambda]\subseteq X[G_\lambda]$ and $H_1$ was defined in $M[G_\lambda]$, it follows that $M[G_\lambda]\models {}^\kappa X[G_\lambda][H_1]\subseteq X[G_\lambda][H_1]$. Also, $j_2(\q_\kappa), q \in X[G_\lambda][H_1]$ and since $j_2(\q_\kappa)$ has size $j_2(\kappa)$, every dense open subset of $j_2(\q_\kappa)$ has a name of the form $j(f)(\kappa,\lambda)$ for some $f:[\kappa]^2\to V_{\kappa+1}$. Using GCH, it follows that there are at most $\kappa^+$-many maximal antichains of $j_2(\q_\kappa)$ in $X[G_\lambda][H_1]$, as counted in $M[G_\lambda]$ and so, we can apply \ref{thm:diagonalisation} to construct an $X[G_\lambda][H_1]$-generic filter $H_2$ below $q$. To show that $H_2$ is also $N[j_2(G_\kappa)]$-generic, let $D\in N[j_2(G_\kappa)]$ be an arbitrary open dense subset of $j_2(\q_\kappa)$. Using the extender representation of $j_2$, we can write $D$ as $(j_2(f_D)(a))_{G_\lambda \ast H_1)}$, for some $a\in [\lambda]^{<\omega}$, $f_D:[\kappa]^{|a|}\to M$, $f_D\in M$. Let
\begin{align*}
D'=\bigcap \{(j_2(f_D)(b))_{G_\lambda\ast H_1}\mid b\in [\lambda]^{<\omega}, j_2(f_D)(b) & \text{ is a } j_2(\p_\kappa)\text{-name for an open} \\ 
& \text{dense subset of }j_2(\q_\kappa)\}.
\end{align*}
$D'$ is well-defined because $j_2(f_D)(a)$ is in the set above. Also, $D'$ is definable from $\lambda$, $j_2(\p_\kappa)$, $G_\lambda\ast H_1$ and $j_2(f_D)$ and so, it is definable in $X[G_\lambda][H_1]$. Also, as $j_2(\q_\kappa)$ is $j_2(\kappa)$-distributive and $D'$ is the intersection of at most $\lambda$-many open dense sets, $D'$ is an open dense set contained in $D$. Since $H_2$ intersects $D'$, it also intersects $D$, therefore $H_2$ is $N[j_2(G_\kappa)]$-generic. Using $H_2$, we can lift $j_2$ to $j_2:M[G_{\kappa+1}]\to N[j_2(G_{\kappa+1})]$, where $j_2(G_{\kappa+1})=G_\lambda\ast H_1\ast H_2$.

\medskip

\subsection*{Constructing a generic for $j_2(\p_{\kl})$.}
The embedding $j_2:M\to N$ is generated by a $(\kappa,\lambda)$-extender, so it has width $\leq\kappa$. By \ref{prop:lift-width}, $j_2:M[G_{\kappa+1}]\to N[j_2(G_{\kappa+1})]$, also has width $\leq\kappa$. Since the forcing $\p_{\kl}$ is $\kappa^+$-strategically closed, if we let $G_{\kl}$ be the part of $G$ that corresponds to $\p_{\kl}$, \ref{thm:transering-generic} implies that the filter $H_3$ generated by $j_2``G_{\kl}$ is $N[j_2(G_\kappa)][H_2]$-generic for $j_2(\p_{\kl})$. It follows that we can lift $j_2$ to $j_2:M[G_{\lambda}]\to N[j_2(G_\lambda)]$, where $j_2(G_\lambda)=G_\lambda\ast H_1\ast H_2\ast H_3$.

\medskip

\subsection*{Constructing a generic for $j_2(\p_{tail})$.}\footnote{The author thanks Yair Hayut for sharing and discussing with him the techniques used in this construction.}
We begin by showing that there is a master condition in $j_2(\p_{tail})$, below which we intend to construct the required generic.

\begin{claim}
There is $q\in j_2(\p_{tail})$ such that $q\leq j(p)$ for all $p\in G_{[\kappa,\lambda)}$.
\end{claim}

\begin{proof}
Since $\p_\lambda$ has the $\lambda$-c.c. and $V\models {}^\lambda M\subseteq M$, the standard arguments show that 
$$V[G_\lambda]\models {}^\lambda M[G_\lambda]\subseteq M[G_\lambda].$$
Hence, the collection $S:=\{j(p)\mid p\in G_\lambda\}$ is in $M[G_\lambda]$ and so, $j_2(S)\in N[k(G_\lambda)]$. Note that $j``G_\lambda\subseteq j_2(S)$ and $j_2(S)$ has size $j_2(\lambda)$ in $N$. 

Roughly, to define $q$ we will consider the coordinate-wise union of conditions in $j_2(S)$, adding their supremum at the top. More precisely, we define in $N[j_2(G_\lambda)]$ a sequence $q$ with domain $(j_2(\lambda),j(\lambda))$ as follows. For each $\alpha\in\dom(q)$, $q(\alpha)$ will be a name for the trivial condition of the $\alpha$-stage of $j_2(\p_{tail})$, unless  $\alpha\in\bigcup_{p\in j_2(S)} \supp(p)$. In this case, using the fact that by elementarity $j_2(S)$ is directed and has size $j_2(\lambda)<\alpha$, we have 
$$\Vdash_{j_2(\p_{tail})\restriction \alpha}\bigcup_{r\in j_2(S)} r(\alpha) \text{ is a bounded subset of } \alpha.$$
Since $j_2(S)$ has size $j_2(\lambda)$ and $\alpha>j_2(\lambda)$ it follows that 
$$\Vdash_{j_2(\p_{tail})\restriction \alpha} \exists x\in \alpha[x=\sup( \bigcup_{r\in j_2(S)} r(\alpha))].$$
By the maximality principle, we can fix a name $\tau_\alpha$ for $x$ and we set 
$$q(\alpha)=\bigcup_{r\in j_2(S)}r(\alpha)\cup \{\tau_\alpha\}.$$
We need to show that $q\in j_2(\p_{tail})$. The support of $q$ is contained in $\bigcup_{r\in j_2(S)}\supp(r)$ and for each $r\in j_2(S)$, $\supp(r)\cap (j_2(\lambda),j(\lambda))$ is an Easton set. As $j_2(S)$ has size $j_2(\lambda)$, it follows that $q$ has Easton support. We also need to show that for all $\alpha\in \supp(q)$, $q(\alpha)$ is a name for a condition in the $\alpha$-stage of $j_2(\p_{tail})$. For each $r\in j_2(S)$, $r(\alpha)$ is a name for a closed bounded subset of $\alpha$ consisting of cardinals which are not ${<}\alpha$-strong. Since $j_2(S)$ is directed, it follows that $\bigcup_{r\in j_2(S)}r(\alpha)$ is forced to be a closed set of singular cardinals, and unbounded in its supremum. So, it remains to show that $\tau_\alpha$ is also forced to be non-${<}\alpha$-strong. To see this, note that by genericity for each $\xi\in (\lambda,j_1(\lambda))$, the supremum of $\bigcup_{r\in S}r(\xi)$ is forced to be greater than $\lambda$. By elementarity, the supremum of $\bigcup_{r\in j_2(S)} r(\alpha)$ is forced to be greater than $j_2(\lambda)$. As $j_2(S)$ has size $j_2(\lambda)$, it follows that $\tau_\alpha$ is a name for a singular cardinal, which in particular is not ${<}\alpha$-strong.

Finally, as $j``G_\lambda\subseteq j_2(S)$, $q$ extends all conditions of the form $(r\restriction (j_2(\lambda),j(\lambda)))_{j_2(G_\lambda)}$, for $r\in j_2(S)$. Thus, $q$ is the required master condition.
\end{proof}

Now that we have the master condition, we force over $V[G]$ to add an $N[j_2(G_\lambda)]$-generic filter $H_4\subseteq j_2(\p_{tail})$ such that $q\in H_4$. The choice of $q$ was so that $q\leq j(p)$ for all $p\in G_{\kl}$. It follows that we can lift $j$ to $j:V[G_\lambda]\to N[j(G_\lambda)]$, where $j(G_\lambda)=G_\lambda\ast H_1\ast H_2\ast H_3\ast H_4$. Obviously, $j$ satisfies the weak $\lambda$-covering property, so for each $\alpha<\lambda$ we can derive from $j$ a fine ultrafilter $U$ on $(\powerset_\kappa\alpha)^{V[G][H_4]}$. We claim that $U\in V[G]$. To see this, note that the stages of $j(\p)$ above $\lambda$ are $\lambda^+$-strategically closed in $N[G_\lambda]$ and so, $(V_\lambda)^{N[j(G_\lambda)]}=(V_\lambda)^{N[G_\lambda]}$. Moreover, $j_2$ was assumed to be a $\lambda$-strongness embedding, so we have $(V_\lambda)^N=(V_\lambda)^M=V_\lambda$ and since the first $\lambda$-stages of $\p$ and $j(\p)$ are the same, $(V_\lambda)^{V[G_\lambda]}=(V_\lambda)^{N[G_\lambda]}$. Hence, forcing with $j_2(\p_{tail})$ over $V[G]$ cannot change $(V_\lambda)^{V[G]}$ and so,
$$(V_\lambda)^{V[G][H_4]}=(V_\lambda)^{V[G]}=V_\lambda[G_\lambda]\subseteq V[G].$$
It follows that $(\powerset_\kappa \alpha)^{V[G][H_4]}=(\powerset_\kappa\alpha)^{V[G]}$ and that $U\in V[G]$. Since $\alpha$ was chosen arbitrarily below $\lambda$, $\kappa$ is ${<}\lambda$-strongly compact in $V[G]$. But also, $\lambda$ can be chosen arbitrarily large below $\delta$ so we have shown that $\kappa$ is $\lessd$-strongly compact in $V[G]$ and the proof is complete.
\end{proof}

To finish the proof, we will show that (2) of \ref{thm:wfsc-definable} holds in $V[G]$, so fix a function $f:\delta\to\delta$ in $V[G]$. Since we use Easton support and $\delta$ is Mahlo, $\p_\delta$ is $\delta$-c.c. and in particular, ${}^\delta \delta$-bounding. This means that there is a function $F:\delta\to\delta$ in $V$ such that for all $\alpha<\delta$, $f(\alpha)<F(\alpha)$. By Claim \ref{claim:strong-sc-v}, we can find in $V[G]$ a cardinal $\kappa<\delta$ which is both $\lessd$-strong for $F$ and $\lessd$-strongly compact. Pick $\lambda>F(\kappa)$ and using \ref{prop:sc-str}, let $j:V[G]\to M$ be an embedding with $\crit(j)=\kappa$, $j(\kappa)>\lambda$, $F\cap V_\lambda=j(F)\cap V_\lambda$, satisfying the $\lambda$-covering property. Since $F(\kappa)<\lambda$, it follows that $j(F)(\kappa)<\lambda$ and thus, $j$ is an embedding which is $j(F)(\kappa)$-strong and $j(F)(\kappa)$-strongly compact. Moreover $\kappa$ is a closure point of $F$ and consequently, of $f$ too. Thus, we have shown that $\delta$ remains Woodin for strong compactnessin $V[G]$.

\medskip

\subsection{Proof of Theorem \ref{thm:first-vopenka}.} 
The first step is to use \ref{vopenka-ind} to show the following result.

\begin{thm}\label{thm:vopenka-eq}
Suppose $\delta$ is a Vop\v enka cardinal. There is a forcing extension inside which $\delta$ remains Vop\v enka, GCH holds, every ground model $\lessd$-supercompact cardinal is preserved and the only $\lessd$-strongly compact cardinals are either the $\lessd$-supercompact cardinals of $V$ or measurable limits of those. 
\end{thm}

Let us temporarily take \ref{thm:vopenka-eq} for granted and see how \ref{thm:first-vopenka} is proved.

Let $V$ be the model induced by \ref{thm:vopenka-eq} and let $\p=\langle \p_\alpha,\dotq_\beta\mid \alpha\leq\delta,\beta<\delta\rangle$ be the following Easton support $\delta$-iteration. Let $\dotq_0$ be a name for $\Add(\omega,1)$. If $\p_\alpha$ has been defined and $\alpha$ was Woodin for strong compactness in $V$, then let $\dotq_\alpha$ name the forcing which shoots a club of singular cardinals below $\alpha$. Otherwise, let $\dotq_\alpha$ name the trivial forcing. Let $G\subseteq \p$ be a $V$-generic filter.

Firstly, we show that there are no Woodin for strong compactness cardinals below $\delta$ in $V[G]$. If $\alpha<\delta$ was Woodin for strong compactness in $V$ then the forcing adds a club of singular cardinals to $\alpha$, thus destroying its Mahloness and in particular, its Woodinised strong compactness. Also, by forcing with $\Add(\omega,1)$ at the first stage, we introduced a closure point at $\omega$. By \ref{thm:cov-approx} the forcing creates no new instances of strong compactness, thus there is no new Woodin for strong compactness cardinal in $V[G]$. 

It remains to show that $\delta$ remains Woodin for strong compactness. As in Claim \ref{claim:strong-not-woodin} in the proof of \ref{thm:first-woodin}, for each $A\subseteq V_\delta$ we can find $\kappa<\delta$ which is $\lessd$-strong for $A$ and not Woodin, and consequently not Woodin for strong compactness. Thus, all stages of $\p$ that are greater than or equal to $\kappa$ are forced to be $\kappa^+$-strategically closed. Using \ref{thm:woodin-indesructibility}, it follows that $\delta$ remains Woodin in $V[G]$.

Now we show that any $\lessd$-supercompact cardinal in $V$ which is not Woodin for strong compactness, remains $\lessd$-supercompact in $V[G]$. For instance, any $\lessd$-supercompact cardinal $\kappa<\delta$ which is not a limit of $\lessd$-supercompacts, is not Woodin for strong compactness ($\kappa$ has a bounded collection of $\lessd$-supercompact cardinals below it, so by the conclusion of \ref{thm:vopenka-eq} it has a bounded collection of $\lessd$-strongly compact cardinals below it and hence, it cannot be Woodin for strong compactness).

Fix such a $\kappa$ and let $\lambda>\kappa$ be a Mahlo cardinal which is not Woodin for strong compactness. Let $j:V\to M$ be a $\lambda$-supercompactness embedding with $\crit(j)=\kappa$. We are going to lift $j$ through $\p$. Since there is trivial forcing at stages $\kappa$ and $\lambda$, we can factorise $\p$ as $\p_\kappa\ast \dotp_{\kl}\ast \dotp_{>\lambda}$, where $\dotp_{\kl}$ is a name for the stages in $(\kappa,\lambda)$ and $\dotp_{>\lambda}$ is a name for the stages greater than $\lambda$.

Since $M$ is closed under $\lambda$-sequences, it has the same Woodin for strong compactness cardinals as $V$ up to $\lambda$. Also, $\lambda$ is not Woodin for strong compactness in $M$. Thus, it is the case that the first $\lambda$-stages of $j(\p_\kappa)$ are the same as those of $\p$ and we can write $j(\p_\kappa)$ as $\p_\lambda\ast \dotp_{tail}$. Using $G_\lambda$ as an $M$-generic filter, $\p_{tail}$ is a $\lambda^+$-strategically closed forcing in $M[G_\lambda]$ and the usual counting arguments, using GCH, show that it has at most $\lambda^+$-many maximal antichains. Also, $\lambda$ is Mahlo, $\p_\lambda\subseteq V_\lambda$ and we are using Easton support, so $\p_\lambda$ is $\lambda$-c.c. As $V\models {}^\lambda M\subseteq M$, it follows that $V[G_\lambda]\models {}^\lambda M[G_\lambda]\subseteq M[G_\lambda]$. Therefore, we can apply \ref{thm:diagonalisation} to construct in $V[G_\lambda]$ an $M[G_\lambda]$-generic filter $H_1\subseteq \p_{tail}$. Since $j``G_\kappa=G_\kappa\subseteq G_\lambda\ast H_1$, we can use Silver's criterion to lift $j$ to $j:V[G_\kappa]\to M[j(G_\kappa)]$, where $j(G_\kappa)=G_\lambda\ast H_1$.

To lift $j$ through $\p_{\kl}=(\dotp_{\kl})_{G_\kappa}$ we notice that since $H_1$ was constructed inside $V[G_\lambda]$, $M[j(G_\kappa)]$ is closed under $\lambda$-sequences in $V[G_\lambda]$. Thus, $j``G_{\kl}\in M[j(G_\kappa)]$, where $G_{\kl}$ is the part of $G$ corresponding to $\p_{\kl}$. We define by induction a sequence $q$ with $\dom(q)=(j(\kappa),j(\lambda))$ as follows. $q(\alpha)$ will be a name for the trivial condition in the $\alpha$-stage of $j(\p_{\kl})$, unless $\alpha\in \bigcup\{\supp(j(p))\mid p\in G_{\kl}\}$. In that case, let $H_\alpha$ be an $M[j_2(G_\kappa)]$-generic filter for $j(\p_{\kl})\restriction \alpha$ and consider the set, 
$$\bigcup\{(j(p)(\alpha))_{j_2(G_\kappa)\ast H_\alpha}\mid p\in G_{\kl}\}.$$
As a union of $\lambda$-many subsets of $\alpha>\lambda$, the above set is bounded and so, it forced by $j(\p_{\kl})\restriction \alpha$ that its supremum is some ordinal less than $\alpha$. Using the maximality principle, we can find a name $\tau_\alpha$ for the supremum and we let 
$$q(\alpha)=\bigcup_{p\in G_{\kl}}j(p)(\alpha) \cup \{\tau_\alpha\}.$$

We need to show that for all $\alpha\in \supp(q)$, $q(\alpha)$ is a name for a condition in the $\alpha$-stage of $j(\p_{\kl})$. For each $p\in G_{\kl}$, $p(\alpha)$ is a name for a closed bounded subset of $\alpha$ consisting of singular cardinals. Since $j``G_{\kl}$ is directed, it follows that $\bigcup_{p\in G_{\kl}}j(p)(\alpha)$ is forced to be a closed set of singular cardinals, and unbounded in its supremum. So, it remains to show that $\tau_\alpha$ is also forced to be singular. To see this, note that $\tau_\alpha$ is forced to be a supremum of a set of size $\lambda$, whose maximum by genericity is greater than $j(\kappa)>\lambda$. It follows that $\tau_\alpha$ is a name for a singular cardinal.

By the definition of $q$, we have $q\leq j(p)$ for all $p\in G_{\kl}$, i.e. it is a master condition. $j(\p_{\kl})$ is (more than) $\lambda^+$-strategically closed in $M[j(G_\kappa)]$ and by a counting argument, we can see that $j(\p_{\kl})$ has at most $\lambda^+$-many maximal antichains in $M[j(G_\kappa)]$, counted in $V[G_\lambda]$. Therefore, the conditions of \ref{thm:diagonalisation} hold and we can construct an $M[G_\lambda][H_1]$-generic filter $H_2\subseteq j(\p_{\kl})$ below $q$. By Silver's criterion we can lift $j$ to $j:V[G_\lambda]\to M[j(G_\lambda)$, where $j(G_\lambda)=G_\lambda\ast H_1\ast H_2$. 

Finally, $\p_{>\lambda}=(\dotp_{>\lambda})_{G_\lambda}$ is $\lambda^+$-distributive and since $j$ had width $\leq\lambda$ we can apply \ref{thm:transering-generic} to transfer $G_{>\lambda}$ to an $M[j(G_\lambda)]$-generic filter $H_3\subseteq j(\p_{>\lambda})$. Then we can lift $j$ to $j:V[G]\to M[j(G)]$, where $j(G)=G_\lambda\ast H_1\ast H_2\ast H_3$. Clearly $j``\lambda\in M[j(G)]$ and so, $j$ is a $\lambda$-supercompactness embedding in $V[G]$. As $\lambda$ can be chosen arbitrarily large, we have shown that $\kappa$ remains $\lessd$-supercompact in $V[G]$.

Therefore, $\delta$ remains a Woodin limit of $\lessd$-supercompact cardinals and it is the first such, since there are no Woodin for strong compactness cardinal below $\delta$ in $V[G]$. 

\begin{proof}[Proof of Theorem \ref{thm:vopenka-eq}]
We begin by forcing the collections of $\lessd$-supercompact and $\lessd$-strongly compact cardinals below $\delta$ to coincide (whenever that is possible). For this, it suffices to adapt the arguments in \cite{apter-laver-compact} to $V_\delta$, for a Vop\v enka cardinal $\delta$. We need a universal Laver function for all the $\lessd$-supercompact cardinals below $\delta$. We omit the proof, since it follows merely from the fact that the definition of the Laver function is uniform, i.e. it does not depend on the particular supercompact cardinal in consideration. For more details, see \cite{apter-laver-compact}.

\begin{lemma}[\cite{apter-laver-compact}]\label{lemma:laver}
There is a function $f:\delta\to V_\delta$ such that whenever $\kappa<\delta$ is $\lessd$-supercompact, $f\restriction \kappa:\kappa\to V_\kappa$ is a Laver function. This means that for any $x\in V_\delta$ and $\lambda<\delta$ such that $x\in H_{\lambda}$, there is a $\lambda$-supercompactness embedding $j:V\to M$ with $j(f)(\kappa)=x$. Moreover, $f$ can be defined so that $f(\alpha)=0$ if $\alpha$ is $\lessd$-supercompact or $\alpha$ is not measurable.
\end{lemma}

Using $f$, we define an Easton support $\delta$-iteration $\p=\langle \p_\alpha,\dotq_\beta\mid \alpha\leq\delta,\beta<\delta\rangle$ along with ordinals $\{\rho_\alpha\mid \alpha<\delta\}$ as follows. Suppose $\p_\alpha$ has been defined, $\rho_\beta<\alpha$ for all $\beta<\alpha$ and $f(\alpha)=\langle \dotq,\sigma\rangle$, where $\dotq$ is a $\p_\alpha$-name for an $\alpha$-directed closed forcing and $\sigma>\alpha$ is regular after forcing with $\p\ast \dotq$. In this case, let $\rho_\alpha=\alpha$, let $\gamma_\alpha=\sup \{\kappa<\alpha\mid \alpha$ is $\lessd$-supercompact$\}$ (and $\gamma_\alpha=\omega$ if there are no $\lessd$-supercompact cardinals below $\alpha$) and define  $\dotq_\alpha=\dotq\ast \dotr_{\gamma_\alpha,\sigma}$, where $\dotr_{\gamma_\alpha,\sigma}$ is a name for the forcing that adds a non-reflecting stationary subset to $\sigma$, consisting of ordinals of cofinality $\gamma_\alpha$. In any other case, $\dotq_\alpha$ is a name for the trivial forcing notion. 

Let $G\subseteq \p$ be a $V$-generic filter and denote $V[G]$ by $W$. It is not hard to see that this iteration satisfies the clauses of Theorem 15 in \cite{brooke-taylor-indestructibility}, so $\delta$ remains Vop\v enka in $W$. Also, the usual proof shows that if $\kappa<\delta$ is $\lessd$-supercompact in $V$, then it remains so in $W$. Note that since $\p$ has plenty of closure points, by \ref{thm:cov-approx} no new $\lessd$-supercompact cardinals are created. We now show that if a cardinal $\kappa$ is $\lessd$-strongly compact in $W$, it was either $\lessd$-supercompact in $V$ or a measurable limit of those.

If neither holds, let $\kappa_0$ be the least regular cardinal greater than $\sup \{\alpha<\kappa\mid \alpha$ is $\lessd$-supercompact$\}$ and $\kappa_1$ the first $\lessd$-supercompact above $\kappa$. Let $j:V\to M$ be a $\kappa_1^+$-supercompactness embedding with $\crit(j)=\kappa_1$ and $j(f)(\kappa)=\langle \dotq,\kappa_1^+\rangle$, where $\dotq$ is a name for the trivial forcing. By the definition of $\p$, $j(\p)$ will have the form $\p_{\kappa_1}\ast \dotq\ast \dotr_{\kappa_0,\kappa_1^+}\ast \dotp_{tail}$, where $\dotp_{tail}$ is a name for the stages above $\kappa$. Thus, $\Vdash_{j(\p)} ``\kappa_1^+$ carried a non-reflecting stationary set of ordinal of cofinality $\kappa_0=j(\kappa_0)$. By \L o\'s' theorem, there are unbounded many $\alpha<\kappa_1$ such that $\p$ forces that $\alpha^+$ has a non-reflecting stationary subset of ordinals of cofinality $\kappa_0$. But this implies that there are no $\lessd$-strongly compact cardinals in $(\kappa_0,\kappa_1)$ in $W$, which contradicts our assumption.

It remains to force with the GCH forcing over $W$ to obtain a universe $W[H]$. It is standard that all measurable and supercompact cardinals are preserved and no new such cardinals are created. Thus, $W[H]$ is the required model.
\end{proof}

\section{Generalisations and questions}

After establishing an identity crisis for the first witness of some large cardinal property, it is customary to try and control the first $n$ witnesses for some $n\in\omega$ or even a proper class of them.

Unlike the difficulties presented in the case of making a class of measurable cardinals coincide with a class of strongly compact cardinals, we show that we can have a proper class of Woodin cardinals coinciding precisely with the Woodin for strong compactness cardinals. The proof is in the spirit of Theorem 2 in \cite{identity-crises-II}.

\begin{thm}\label{thm:many-woodin}
Suppose there is a proper class of Vop\v enka cardinals and that GCH holds. Then we can construct a model in which there is a proper class of Woodin for strong compactness cardinals which coincides with the class of Woodin cardinals.
\end{thm}

\begin{proof}
If there is an inaccessible limit of Vop\v enka cardinals and let $\sigma$ be the least such and otherwise, let $\sigma=\ord$. Let $\langle \delta_\alpha\mid \alpha\in \sigma\rangle$ be an increasing enumeration of the Vop\v enka cardinals below $\sigma$. For each $\alpha$, let $\p_\alpha$ denote the Easton support $\delta_\alpha$-iteration defined as in the proof of Theorem \ref{thm:first-woodin}, destroying Woodin cardinals in the interval $(\gamma_\alpha,\delta_\alpha)$, where $\gamma_\alpha=\sup\{\xi<\delta_\alpha\mid \xi$ is Vop\v enka$\}$.

Let $\p$ be the Easton product $\prod_{\alpha<\sigma} \p_\alpha$. In case $\sigma=\ord$, the standard arguments show that $\p$ preserves ZFC. We argue that after forcing with $\p$, each $\delta_\alpha$ is Woodin for strong compactness and not Vop\v enka, and these are the only Woodin cardinals below $\sigma$. Fix some $\alpha<\sigma$ and factorise the forcing as 
$$\prod_{\beta<\alpha} \p_\beta \times \p_\alpha \times \prod_{\beta>\alpha} \p_\beta.$$
Let $G=G_{<\alpha}\times G_\alpha\times G_{>\alpha}$ be a $V$-generic filter for $\p$. The forcing $\prod_{\beta>\alpha}\p_\beta$ is $\delta_\alpha^+$-distributive, so $\delta_\alpha$ remains Vop\v enka in $V[G_{>\alpha}]$. As in the proof of \ref{thm:first-woodin}, ading the generic filter $G_\alpha\subseteq\p_\alpha$ makes $\delta_\alpha$ Woodin for strong compactness and not Vop\v enka in $V[G_{>\alpha}][G_\alpha]$, while killing all Woodin cardinals in $(\gamma_\alpha,\delta_\alpha)$.  Finally, $\prod_{\beta<\alpha}\p_\beta$ is small compared to $\delta_\alpha$, so in $V[G]$, $\delta_\alpha$ is still Woodin for strong compactness. 

We claim that if $\rho<\sigma$ is a Woodin cardinal in $V[G]$, then $\rho=\delta_\alpha$ for some $\alpha<\sigma$. Otherwise, there is $\alpha$ such that $\delta_\alpha<\rho<\delta_{\alpha+1}$ and then the forcing $\p_{\alpha+1}$ shot a club at $\rho$ destroying its Woodinness. The rest of $\p$ cannot change this fact, so $\rho$ is not Woodin in $V[G]$ which is absurd. 

Therefore, the universe $W=(V_\sigma)^{V[G]}$ if $\sigma$ is inaccessible, or $V[G]$ if $\sigma=\ord$, has a proper class of Woodin for strong compactness cardinals which coincide with the Woodin cardinals. 
\end{proof}

Using similar arguments, we can show that the dual holds too.

\begin{thm}\label{thm:many-vopenka}
Suppose there is a proper class of Vop\v enka cardinals and that GCH holds. Then we can construct a model where there is a proper class of Woodin for strong compactness cardinals which coincide with the Woodin limit of supercompact cardinals.
\end{thm}

\begin{proof}
As previously, let $\sigma$ be the first inaccessible limit of Vop\v enka cardinals if there is such a cardinal, or let $\sigma=\ord$ otherwise. Let $\langle \delta_\alpha\mid \alpha<\sigma\rangle$ be an increasing enumeration of the Vop\v enka cardinals below $\sigma$. For each $\alpha<\sigma$, we define a two-step iteration $\p_\alpha\ast \dotq_\alpha$ as follows. $\p_\alpha$ is the forcing defined in the proof of \ref{thm:first-vopenka} with the following changes:
\begin{enumerate}
	\item the Laver function we use has domain $(\gamma_\alpha,\delta_\alpha)$, where $\gamma_\alpha=\sup \{\xi<\delta_\alpha\mid \xi$ is Vop\v enka$\}$,
	\item if $\xi$ is a non-trivial stage and there are no ${<}\delta_\alpha$-supercompact cardinals below $\xi$, then the non-reflecting stationary set added consists of ordinals of cofinality $\gamma_\alpha$.
\end{enumerate}
Then, let $\dotq_\alpha$ be a name for the forcing which destoys all Woodin for strong compactness cardinal in $(\gamma_\alpha,\delta_\alpha)$, as in \ref{thm:first-vopenka}. In fact, the proof of \ref{thm:first-vopenka} shows that $\p_\alpha\ast \dotq_\alpha$ preserves the fact that $\delta_\alpha$ is a Woodin limit of $\lessd_\alpha$ supercompact cardinals, while killing all Woodin for strong compactness cardinals in $(\gamma_\alpha,\delta_\alpha)$.

Now, if we let $\p$ be the Easton product $\prod_{\alpha<\sigma} (\p_\alpha\ast \dotq_\alpha)$, then we can argue as in \ref{thm:many-woodin} to show that for all $\alpha<\sigma$, $\delta_\alpha$ is a Woodin limit of ${<}\delta_\alpha$-supercompact cardinals and so, Woodin for strong compactness and there are no other Woodin for strong compactness cardinals. Thus, the universe $W=(V_\sigma)^{V[G]}$ if $\sigma$ is inaccessible, or $V[G]$ if $\sigma=\ord$, has a proper class of Woodin limits of supercompact cardinals which coincide with the Woodin for strong compactness cardinals.  
\end{proof}

We finish by mentioning some open questions. For both \ref{thm:first-woodin} and \ref{thm:first-vopenka} we assumed the existence of a Vop\v enka cardinal. It it still open whether the assumptions in both theorems can be reduced.

\begin{qst}
Can we reduce the large cardinal assumptions of Theorem \ref{thm:first-woodin}? That is to start with a Woodin for strong compactness cardinal instead of a Vop\v enka cardinal.
\end{qst}

\begin{qst}
Can we reduce the large cardinal assumptions of Theorem \ref{thm:first-woodin}? That is to start with a Woodin limit of supercompacts instead of a Vop\v enka cardinal.
\end{qst}

Moreover, the model induced in \ref{thm:many-woodin} and \ref{thm:many-vopenka} has no inaccessible limit of Woodin for strong compactness cardinals. 

\begin{qst}
Can we prove \ref{thm:many-woodin} or \ref{thm:many-vopenka} without any restrictions on the large cardinal structure?
\end{qst}

Since the assumptions of \ref{thm:first-woodin} include GCH, we also ask the following.

\begin{qst}
Can we force GCH in the presence of a Woodin for strong compactness cardinal? Even more, can we realise Easton functions in the presence of a Woodin for strong compactness cardinal?
\end{qst}

Note that the same question is still open for strongly compact cardinals, i.e. it is still unknown whether we can control the continuum or even force GCH in the presence of a strongly compact cardinal without assuming supercompactness.

Finally, as for strongly compact cardinals, the consistency strength of Woodin for strong compactness cardinals remains unclear. A lower bound is a proper class of strongly compact cardinals and an upper bound is a Woodin limit of supercompact cardinals, which lies below an extendible cardinal.

\begin{qst}
What is the exact consistency strength of a Woodin for strong compactness cardinal?
\end{qst}

\subsection*{Acknowledgements}
The results presented here are part of the author's doctoral thesis, supported by the UK Engineering and Physical Sciences Research Council (EPSRC) DTA studentship 14 EP/M506473/1. The author wishes to thank his supervisors Andrew Brooke-Taylor and Philip Welch for their guidance and support. The author is also grateful to the anonymous referee for their careful reading, for finding certain mistakes in the original preprint of this article and for their helpful suggestions, which considerably improved the presentation of the results. 

\bibliographystyle{plain}
\bibliography{math_refs}

\begin{thebibliography}{10}

\bibitem{apter-laver-compact}
Arthur~W. Apter.
\newblock Laver indestructibility and the class of compact cardinals.
\newblock {\em J. Symbolic Logic}, 63(1):149--157, 1998.

\bibitem{identity-crises-I}
Arthur~W. Apter and James Cummings.
\newblock Identity crises and strong compactness.
\newblock {\em J. Symbolic Logic}, 65(4):1895--1910, 2000.

\bibitem{identity-crises-II}
Arthur~W. Apter and James Cummings.
\newblock Identity crises and strong compactness. {II}. {S}trong cardinals.
\newblock {\em Arch. Math. Logic}, 40(1):25--38, 2001.

\bibitem{identity-crises-III}
Arthur~W. Apter and Grigor Sargsyan.
\newblock Identity crises and strong compactness. {III}. {W}oodin cardinals.
\newblock {\em Arch. Math. Logic}, 45(3):307--322, 2006.

\bibitem{brooke-taylor-indestructibility}
Andrew~D. Brooke-Taylor.
\newblock Indestructibility of {V}op\v enka's {P}rinciple.
\newblock {\em Arch. Math. Logic}, 50(5-6):515--529, 2011.

\bibitem{cody-easton-woodin}
Brent Cody.
\newblock Easton's theorem in the presence of {W}oodin cardinals.
\newblock {\em Arch. Math. Logic}, 52(5-6):569--591, 2013.

\bibitem{cummings-handbook}
James Cummings.
\newblock Iterated forcing and elementary embeddings.
\newblock In {\em Handbook of set theory. {V}ols. 1, 2, 3}, pages 775--883.
  Springer, Dordrecht, 2010.

\bibitem{woodin-cardinals-forcing}
S.~{Dimopoulos}.
\newblock {Woodin cardinals and forcing}.
\newblock {\em ArXiv e-prints}, November 2017.

\bibitem{hamkins-largecardinals}
Joel~David Hamkins.
\newblock Forcing and large cardinals.
\newblock manuscript.

\bibitem{hamkins-approximation}
Joel~David Hamkins.
\newblock Extensions with the approximation and cover properties have no new
  large cardinals.
\newblock {\em Fund. Math.}, 180(3):257--277, 2003.

\bibitem{jech}
Thomas Jech.
\newblock {\em Set theory}.
\newblock Springer Monographs in Mathematics. Springer-Verlag, Berlin, 2003.
\newblock The third millennium edition, revised and expanded.

\bibitem{kanamori}
Akihiro Kanamori.
\newblock {\em The higher infinite}.
\newblock Springer Monographs in Mathematics. Springer-Verlag, Berlin, second
  edition, 2009.
\newblock Large cardinals in set theory from their beginnings, Paperback
  reprint of the 2003 edition.

\bibitem{magidor-identity-crises}
Menachem Magidor.
\newblock How large is the first strongly compact cardinal? or {A} study on
  identity crises.
\newblock {\em Ann. Math. Logic}, 10(1):33--57, 1976.

\bibitem{perlmutter-2015}
Norman~Lewis Perlmutter.
\newblock The large cardinals between supercompact and almost-huge.
\newblock {\em Arch. Math. Logic}, 54(3-4):257--289, 2015.

\end{thebibliography}

\end{document}